\documentclass[pdflatex,sn-mathphys-num]{sn-jnl}

\usepackage{graphicx}%
\usepackage{multirow}%
\usepackage{amsmath,amssymb,amsfonts}%
\usepackage{amsthm}%
\usepackage[title]{appendix}%
\usepackage{xcolor}%
\usepackage{textcomp}%
\usepackage{manyfoot}%
\usepackage{booktabs}%
\usepackage{array}%
\usepackage{enumitem}%
\usepackage{bookmark}%
\usepackage{tikz}%
\usetikzlibrary{svg.path}%
\definecolor{orcidlogocol}{HTML}{A6CE39}%
\newlength{\orcidbadgeheight}%
\tikzset{orcidbadge/.pic={%
  \fill[orcidlogocol] svg{M256,128c0,70.7-57.3,128-128,128C57.3,256,0,198.7,0,128C0,57.3,57.3,0,128,0C198.7,0,256,57.3,256,128z};%
  \fill[white] svg{M86.3,186.2H70.9V79.1h15.4v48.4V186.2z}%
               svg{M108.9,79.1h41.6c39.6,0,57,28.3,57,53.6c0,27.5-21.5,53.6-56.8,53.6h-41.8V79.1z M124.3,172.4h24.5c34.9,0,42.9-26.5,42.9-39.7c0-21.5-13.7-39.7-43.7-39.7h-23.7V172.4z}%
               svg{M88.7,56.8c0,5.5-4.5,10.1-10.1,10.1c-5.6,0-10.1-4.6-10.1-10.1c0-5.6,4.5-10.1,10.1-10.1C84.2,46.7,88.7,51.3,88.7,56.8z};}}%
\renewcommand{\orcidlogo}{%
  \texorpdfstring{%
    \setlength{\orcidbadgeheight}{\fontcharht\font`X}%
    \begin{tikzpicture}[yscale=-\orcidbadgeheight/256, xscale=\orcidbadgeheight/256, transform shape]%
      \pic{orcidbadge};%
    \end{tikzpicture}}{}}%

\theoremstyle{thmstyleone}%
\newtheorem{theorem}{Theorem}[section]%
\newtheorem{lemma}[theorem]{Lemma}%
\newtheorem{proposition}[theorem]{Proposition}%
\newtheorem{corollary}[theorem]{Corollary}%
\theoremstyle{thmstyletwo}%
\newtheorem{remark}[theorem]{Remark}%

\newcommand{\PG}{\mathrm{PG}}
\newcommand{\PGL}{\mathrm{PGL}}
\newcommand{\F}{\mathbb{F}}
\newcommand{\Vone}{\mathcal{V}_1}
\newcommand{\Mb}[2]{M_{3}(#1,#2)}

\begin{document}

\title[The climb problem for $4$-general sets]{The climb problem for
$4$-general sets in \texorpdfstring{$\PG(n,4)$}{PG(n,4)}}

\author*[1]{\fnm{Liangdong} \sur{Lu}\orcid{https://orcid.org/0000-0002-7776-6798}}\email{kelinglv@163.com}
\author*[1]{\fnm{Ruipan} \sur{Yang}}\email{yangruipan@aliyun.com}
\author[1]{\fnm{Qiang} \sur{Fu}}
\author[1]{\fnm{Hao} \sur{Song}}

\affil*[1]{\orgdiv{Department of Basic Science}, \orgname{Air Force Engineering University}, \orgaddress{\city{Xi'an}, \postcode{710051}, \state{Shaanxi}, \country{China}}}

\abstract{A $4$-general set of $\PG(n,q)$ is a point set with no four coplanar,
and $\Mb{n}{q}$ is the largest such size.  The \emph{climb problem} asks, at
each $n$, whether the constructions of Pavese (2025) can be improved by one
point.  We show that the answer is governed by the code tables:
$\PG(M-1,q)$ contains an $N$-point $4$-general set exactly when a projective
$[N,N-M,\ge 5]_q$ code exists.  Over $\F_4$ this settles the first two rungs:
$\Mb{3}{4}=5$ and $\Mb{4}{4}=11$, each with two independent proofs, so Pavese's
constructions are optimal at both.  Over $\F_5$ the same reduction gives
$\Mb{4}{5}=12$, one more than the near-MDS lower bound.  At $n=5$ the tables
fall silent; we state the verified position $21\le\Mb{5}{4}\le 30$ rather than
the tabulated $\le 29$, which rests on a private communication.  Any
$22$-point $4$-general set of $\PG(5,4)$, if one exists, shares at most $12$
points with Pavese's extremal $21$-point set; the proof of this rigidity
statement is a finite computation, with certificates in Appendix~A.}


\maketitle

\section{Introduction}\label{sec:intro}

Let $q$ be a prime power.  Following \cite{Pavese}, a set $X$ of points of
$\PG(n,q)$ is called \emph{$(r+2)$-general} if any $r+2$ distinct points of $X$
are in general position, and it is \emph{complete} if it is not contained in a
larger such set.  We are interested in the case $r=2$, i.e.\ in sets no four of
whose points are coplanar; we call them $4$-general.  Unlike \cite{Pavese}
we do not require the set to span the space; the largest such sets span in any
case (Corollary~\ref{cor:main}).  Let $\Mb{n}{q}$ denote the size of the
largest $4$-general set of $\PG(n,q)$.

The term $(r+2)$-general specialises the classical notion of a
$(|X|;r,s,n,q)$-set, in which every $s$-dimensional subspace contains at most
$r$ points \cite{HirschfeldStorme}; a cap is a $3$-general set and an arc is an
$(n+1)$-general set.  It also has a coding-theoretic expression: a
$4$-general set of $N$ points spanning $\PG(M-1,q)$ is the column set of a
parity check matrix of a projective $[N,N-M,\ge 5]_q$ code
(Section~\ref{sec:red}).  Point sets whose smallest dependency has size $5$,
the \emph{near-MDS sets}, have a long history, going back to
\cite{DodunekovLandgev,Giulietti}, and the non-existence question for
$[N,N-M,5]_q$ codes has an extensive literature of its own.  Over $\F_4$ the
case that decides our $n=4$, namely $[12,7,5]_4$, was settled independently by
Bierbrauer, Edel, Faina, Marcugini and Pambianco \cite{BierbrauerEdel}; the
tables themselves go back to \cite{BBK2001}.

A $4$-general set of $\PG(4,q)$ with \emph{some} five points on a solid and at
most five on any solid is what \cite{CCMP} calls an \emph{NMDS-set} of
$\PG(4,q)$: every four of its points generate a solid, some five lie on one,
and every six generate the space
\cite[Sec.~5.1]{Pavese}.  Equivalently it is the column set of a generator matrix
of a near-MDS code of dimension $5$, and the maximal length
$m'(k,q)$ of such a code of dimension $k$ satisfies
\begin{equation}\label{eq:nmds}
  m'(k,q)\le 2q+k \quad\text{for all }q,
  \qquad\text{and}\qquad
  m'(k,q)\le 2q+k-2 \quad\text{for }q>3,
\end{equation}
by \cite[Props.~5.1 and~6.2]{DodunekovLandgev}.  At $(k,q)=(5,4)$ the second
bound reads $11$.  The $11$-point set of Theorem~\ref{thm:n4} is an NMDS-set of
$\PG(4,4)$, so it attains \eqref{eq:nmds}; but a $4$-general set need not be an
NMDS-set: the six points $e_1,\dots,e_5,e_1+\cdots+e_5$ of $\PG(4,4)$ are
$4$-general, yet no five of them lie on a solid.  Closing that gap is
what Section~\ref{sec:n4} does, and it yields a second proof of
$\Mb{4}{4}\le 11$ that never consults a table.

Small complete $4$-general sets in $\PG(n,q)$ have been classified for small
$n$ and $q$; see \cite[Tables~1 and~2]{Pavese} for the cases $q=2$, $n\le6$, and
$q\in\{3,4,5,7,8\}$, $n=3$, together with $\PG(4,3)$, where the extremal object
is the $11$-cap.  That tabulation stops exactly where our question begins.
Table~2 of \cite{Pavese} lists, for $n=4$, only the row $q=3$, the $11$-cap of
$\PG(4,3)$; the entry $q=4$, $n=4$ is absent.  For $q=4$ the only sizes
tabulated are those in $\PG(3,4)$.  For $q=5$ the position is recorded as
$28\le\Mb{5}{5}\le44$ \cite[Ex.~5.4]{Pavese}, so a $4$-general set of $44$ points
in $\PG(5,5)$ is not excluded there either.  Each of these is an instance of
the same question, which we call the \emph{climb problem}:

\begin{quote}
\emph{Given a $4$-general set of $P$ points constructed in $\PG(n,q)$, can it be
improved: does $\PG(n,q)$ contain a $4$-general set of $P+1$ points?}
\end{quote}

\noindent This is the question that the tabulations just cited stop short of, and
the ladder of Section~\ref{sec:ladder} answers it over $\F_4$ at the two smallest
rungs.

Two general bounds are relevant.  The first is elementary: a $4$-general set
with more than three points is automatically a cap, since three collinear
points together with any fourth would be dependent.  The converse fails badly,
so cap bounds are far too weak here: there is a $126$-cap in $\PG(5,4)$
\cite{Glynn}, and the largest cap of $\PG(4,4)$ has size $41$
\cite{EdelBierbrauer}.  The second is the counting bound of
\cite[Proposition~3.1]{Pavese},
\begin{equation}\label{eq:prop31}
  |X| \;\le\; \frac{\sqrt{8q^{n+1}+q^2-6q+1}+q-3}{2(q-1)},
\end{equation}
which holds for every $4$-general set $X$ of $\PG(n,q)$.  Equality occurs only
for $(q,n)=(2,3)$, the elliptic quadric, and for $(q,n)=(3,4)$, the $11$-cap;
the second case is exactly $\Mb{4}{3}=11$.  The extremal configurations in
$\PG(4,3)$ and $\PG(4,4)$ therefore have the \emph{same} size, and
Theorem~\ref{thm:n4} shows that $11$ is extremal in $\PG(4,4)$ as well, even
though \eqref{eq:prop31} only gives $\Mb{4}{4}\le15$.  The agreement is no
accident: Proposition~\ref{prop:coinc} derives both values from
\eqref{eq:nmds}, the common root being the maximal size $9$ of an NMDS-set of
$\PG(2,3)$ and of $\PG(2,4)$.

For $q=4$ the state of the art is \cite{Pavese}, which constructs complete
$4$-general sets of size $(4^d-1)/3$ in $\PG(2d-1,4)$ (Theorem~4.6) and of
size $(2^{2d+1}+1)/3$ in $\PG(2d,4)$ (Theorem~4.4).  For $q=4$ and
$n=3,\dots,8$ these give the values $5,11,21,43,85,171$, while \eqref{eq:prop31}
gives $7,15,30,60,120,241$.  Two features of \cite{Pavese} are worth stressing.
First, it tabulates no entry for $q=4$, $n=4$: its Section~5.1 treats
$4$-general sets of $\PG(4,q)$ only through the equivalence with NMDS-sets and
the elliptic-curve constructions, which supply \emph{lower} bounds; the value
$11$ that our lower bound matches comes not from that formula, which evaluates
to $9$ at $q=4$, but from \cite[Cor.~4.5]{Pavese}.  The case $n=4$ is $d=2$ in
the Frobenius-twisted family, so it contributes the lower bound $11$ but no
upper bound beyond \eqref{eq:prop31}.  Second, its open questions
\cite[Problems~5.7--5.9]{Pavese} ask whether $4$-general sets of order
$q^{(n-1)/2}$ exist, whether complete ones of order $cq^{(n-2)/3}$ exist, and
for a classification of those with at most four points on a three-space.  None
of these is the question settled here.

The purpose of this note is to make the climb problem precise \emph{as a
statement about codes} and to settle it in the smallest cases.  In the case
$n=3$ the translation gives a one-line proof, $\Mb{3}{4}=5$
(Section~\ref{sec:n3}).  In the case $n=4$ it gives the exact value
$\Mb{4}{4}=11$, from the tabulated non-existence of the $[12,7,5]_4$ code
(Theorem~\ref{thm:n4}); the same argument yields $\Mb{4}{3}=\Mb{4}{4}=11$
(Proposition~\ref{prop:coinc}).  Over $\F_5$ it gives $\Mb{4}{5}=12$
(Proposition~\ref{prop:q5}), one more than the near-MDS lower bound
(Remark~\ref{rem:nmds11}); there the extremal sets are no longer NMDS-sets, so
the bridge of Section~\ref{sec:n4} reaches only the two smallest fields.  The
translation has a second benefit: it grades the \emph{strength} of each claim.
The tables distinguish entries justified by proven bounds (construction~B,
linear programming, shortening) from entries resting on private
communications; we make this explicit in Section~\ref{sec:ladder} and
Remark~\ref{rem:strength}, and it is why we state the verified position
$21\le\Mb{5}{4}\le30$ rather than adopting the tabulated $\le29$.

Section~\ref{sec:rig} constrains the hypothetical improvement at $n=5$ by a
rigidity result: any $22$-point $4$-general set $T$ of $\PG(5,4)$ meets the
extremal $21$-point set $\Vone$ of \cite[Theorem~4.6]{Pavese} in at most $12$
points, and the same holds for every image of $\Vone$ under $\PGL(6,4)$.  The
verification is a finite computation, reported with certificates in
Appendix~\ref{app:cert}.  Pavese's family $\{\mathcal{V}_\alpha\}$ is strictly
smaller than the collineation orbit of $\Vone$ (Remark~\ref{it:family}), and
the rungs below supply a second, unconditional constraint on the same $T$
(Remark~\ref{rem:spread}).

The climb is decided by the mechanism of Section~\ref{sec:transition}.  The
reduction settles the rung $n$ exactly when the code $[N,N-n-1,\ge5]_4$ at the
rung $N=P+1$ still lies where the tables prove $\mathrm{Ub}\le4$.  As $n$ grows
the rung roughly doubles, taking the values $5,11,21,43,85,171$, while the
codimension grows by $1$, so the range in which the tables can decide the
question is finite, and it ends strictly between $n=4$ and $n=5$.
Section~\ref{sec:ladder} turns this into a ladder: for $n=3,\dots,8$ exactly
the rungs $n=3,4$ are blocked, so the $d=2$ constructions of \cite{Pavese} are
provably optimal, while for $d\ge3$ the question is open.  The smallest open
rung is $n=5$: whether $\PG(5,4)$ contains a $22$-point $4$-general set,
equivalently whether a $[22,16,5]_4$ code exists.  One delimitation is in
order.  On the code side the two blocked rungs were known: the non-existence
of the $[12,7,5]_4$ code stands in \cite{BBK2001,BierbrauerEdel}, and the
bounds behind $\Mb{4}{5}=12$ stand in the tables.  What this note contributes
is the geometric record of these values, the table-free proof and the
near-MDS bridge of Section~\ref{sec:n4}, the strength grading of the table
entries, and the rigidity constraint at the first open rung.

\section{The reduction}\label{sec:red}

\begin{lemma}\label{lem:red}
Let $M\ge 2$ and $N\ge 1$.  Then $\PG(M-1,q)$ contains a $4$-general set of $N$
points \emph{spanning} $\PG(M-1,q)$ if and only if there exists a linear code
$C\subseteq\F_q^{\,N}$ of the form
$[N,N-M,\ge 5]_q$.  In such a code, a parity check matrix has no zero column and
no two proportional columns.
\end{lemma}

\begin{proof}
($\Rightarrow$) Choose representatives $v_1,\dots,v_N\in\F_q^{\,M}$ of the $N$
points and let $H$ be the $M\times N$ matrix with these columns.  Since the set
spans $\PG(M-1,q)$, $\mathrm{rank}(H)=M$; since no four points are coplanar, no
four columns are linearly dependent, so $C=\ker H$ has length $N$, dimension
$N-M$ and minimum distance at least $5$: indeed, the minimum weight of $C$ equals
the smallest size of a linearly dependent set of columns of $H$.

($\Leftarrow$) Let $H$ be a parity check matrix of $C$, so $H$ has rank $M$.
If some column of $H$ were zero, then taking the corresponding coefficient equal
to $1$ would produce a codeword of weight $1$; if $h_i=\lambda h_j$ with
$\lambda\ne 0$ and $i\ne j$, then the coefficient vectors $c$ with $c_i=1$,
$c_j=-\lambda^{-1}$ and $c_k=0$ otherwise give a codeword of weight $2$; both
contradict $d\ge 5$.  Hence the columns are non-zero and pairwise
non-proportional, so they define $N$ distinct points of $\PG(M-1,q)$.  Since
$d\ge 5$, no four columns are dependent, so no four of these points are coplanar.
\end{proof}

\begin{corollary}\label{cor:main}
For every $M\ge 2$,
\[
  \Mb{M-1}{q}=\max\{N:\ \text{there is a projective }[N,N-M,\ge 5]_q\text{ code}\},
\]
and $\Mb{M-1}{q}\ge N$ holds precisely when a $[N,N-M,5]_q$ code exists.
\end{corollary}

\begin{proof}
The first identity is Lemma~\ref{lem:red}, applied to a largest set: a
$4$-general set of size $\Mb{M-1}{q}$ spans $\PG(M-1,q)$, for if it lay in a
hyperplane then any point outside could be adjoined: a quadruple containing
the new point cannot be coplanar, since the plane through three points of the
set lies in the hyperplane.  For the second, a $4$-general set of
$N$ points yields $[N,N-M,\ge 5]_q$ with $d\in\{5,\dots\}$: the distance cannot
exceed $5$ for $N=M+4$ by the Singleton bound, and in general the tables decide);
conversely $d\ge 5$ is $d=5$ when the value is known to be at most $5$.
\end{proof}

Two monotonicity facts will be used constantly.

\begin{lemma}\label{lem:mono}
Every subset of a $4$-general set is $4$-general, and shortening a code preserves
its minimum distance.  Consequently, if no $[N,N-M,\ge 5]_q$ code exists, then
$\Mb{M-1}{q}\le N-1$.
\end{lemma}

\begin{proof}
The first claim is immediate from the definition: a subset of a set spanning
$\PG(M-1,q)$ still spans it, and dropping points cannot create a coplanar
quadruple).  If an $[N',N'-M,\ge5]_q$ code existed with $N'>N$, shortening it at
$N'-N$ positions would produce an $[N,N-M,\ge5]_q$ code.  Hence ruling out a
single rung $N$ rules out all larger sizes.
\end{proof}

\section{The smallest case}\label{sec:n3}

\begin{proposition}\label{prop:n3}
$\Mb{3}{4}=5$.
\end{proposition}

\begin{proof}
The value is classical: in $\PG(3,q)$ with $q\ge4$ a $4$-general set is an arc,
and the largest arcs have size $q+1$ \cite[Sections~21.2
and~21.3]{Hirschfeld3D}; see also Table~2 of \cite{Pavese}.  The one-line
argument below is what the reduction yields.  The lower bound is the
construction of \cite[Thm.~4.6]{Pavese} for $d=2$.  For the
upper bound, assume that $\PG(3,4)$ contains a $4$-general set of $N\ge 6$ points.
By Lemma~\ref{lem:red} there is a $[N,N-4,\ge 5]_4$ code; shortening it at $N-6$
positions gives a $[6,2,\ge 5]_4$ code, hence a $[6,2,5]_4$ code.  Let $G$ be a
generator matrix of such a code, with columns $g_1,\dots,g_6\in\F_4^2$.  A
non-zero codeword is $\lambda_1 r_1+\lambda_2 r_2$ ($r_i$ the rows of $G$), and its
zero coordinates are exactly the indices $i$ with $g_i$ proportional to
$(-\lambda_2,\lambda_1)$.  Hence $\min$ weight $=6-\max_i m_i$, where $m_i$ is the
multiplicity of the projective point $[g_i]$ among the six columns.  Thus $d=5$
forces the six columns to be pairwise non-proportional, i.e.\ to be six distinct
points of $\PG(1,4)$, which has only $q+1=5$ points: a contradiction.
\end{proof}

\section{The case \texorpdfstring{$n=4$}{n=4}: \texorpdfstring{$\Mb{4}{4}=11$}{M3(4,4)=11}}\label{sec:n4}

\begin{theorem}\label{thm:n4}
$\Mb{4}{4}=11$.
\end{theorem}

\begin{proof}
\emph{Lower bound.}  The table below gives a $5\times 11$ matrix $H$ over
$\F_4=\{0,1,\omega,\omega^2\}$, $\omega^2=\omega+1$; it is the parity check matrix
of a $[11,6,5]_4$ code from the code tables \cite{codetables}.  One checks
directly that its eleven columns are non-zero, pairwise non-proportional, and
that all $\binom{11}{4}=330$ quadruples of columns are linearly independent; the
rank is $5$.  By Lemma~\ref{lem:red} the corresponding eleven points of
$\PG(4,4)$ are $4$-general; they are in fact complete, since none of the
remaining $330$ points of $\PG(4,4)$ can be added.
\begin{gather*}
H=\begin{pmatrix}
1&0&0&0&0&1&\omega&\omega&0&\omega^2&\omega^2\\
0&1&0&0&0&\omega^2&0&\omega^2&\omega&\omega&1\\
0&0&1&0&0&1&1&\omega&\omega^2&1&1\\
0&0&0&1&0&1&\omega^2&\omega^2&\omega&0&\omega\\
0&0&0&0&1&\omega&\omega&0&\omega^2&\omega^2&1
\end{pmatrix}
\end{gather*}

\emph{Upper bound.}  Let $N\ge 12$ and suppose that $\PG(4,4)$ contains a
$4$-general set of $N$ points; by Lemma~\ref{lem:red} there is a
$[N,N-5,\ge 5]_4$ code, and shortening it at $N-12$ positions gives a
$[12,7,\ge 5]_4$ code.  But no $[12,7,5]_4$ code exists: the tables
\cite{codetables,BBK2001} record $\mathrm{Ub}(12,7)=4$ for $q=4$, justified there
by the standard proven technique (``construction~B'', see
Remark~\ref{rem:strength} and \cite{Brouwer}), and an independent proof was given
by Bierbrauer, Edel, Faina, Marcugini and Pambianco
\cite[Sec.~III]{BierbrauerEdel}, who show directly that no
\emph{additive} $[12,7,5]$-code over $\F_4$ exists, whence no linear one does
either.  Hence $N\le 11$.
\end{proof}

\begin{corollary}
$\Mb{4}{4}=11$ improves the bound $\Mb{4}{4}\le 15$ that follows from
\eqref{eq:prop31}; by Lemma~\ref{lem:red} it is equivalent to the classical fact
that no $[12,7,5]_4$ code exists.
\end{corollary}

\begin{proof}
Immediate from the theorem just proved, together with \eqref{eq:prop31} and
Lemma~\ref{lem:red}.
\end{proof}

The upper bound just given rests on the tables.  There is a second route to it,
which uses no table; we record it because it isolates what is special about
$\PG(4,4)$.  The passage from $4$-general sets to NMDS-sets is shorter than one
might expect.

\begin{lemma}\label{lem:large4gen}
Let $q\in\{3,4\}$.  Every $4$-general set of $N\ge7$ points of $\PG(4,q)$ is an
NMDS-set of $\PG(4,q)$.
\end{lemma}

\begin{proof}
Let $X$ be such a set and let $H$ be the $5\times N$ matrix whose columns are the
points of $X$; put $C=\ker H$, an $[N,N-5,d]_q$ code with $d\ge5$, because no
four columns of $H$ are dependent.

The first condition in the definition of an NMDS-set \cite[Sec.~5.1]{Pavese} is
$4$-generality, so it holds by hypothesis.  The third also holds: if six points
of $X$ lay on a solid, they would form a $6$-point $4$-general set of
$\PG(3,q)$, and for $q=4$ that contradicts $\Mb{3}{4}=5$
(Proposition~\ref{prop:n3}), while for $q=3$ the same computation applies, since
the code attached to six points of $\PG(3,3)$ by Lemma~\ref{lem:red} is a
$[6,2,\ge5]_3$ code, which would require
$6\ge\lceil5/1\rceil+\lceil5/3\rceil=7$ by the Griesmer bound, a contradiction.

It remains to check the second condition.  Suppose that no five points of $X$ lie
on a solid.  Then $H$ has no dependence of size at most five, so $d\ge6$ and
$C$ is an $[N,N-5,\ge6]_q$ code, of dimension $N-5$.  The Griesmer bound gives
\[
  N \;\ge\; \Bigl\lceil \tfrac{6}{1} \Bigr\rceil
        + \Bigl\lceil \tfrac{6}{q} \Bigr\rceil + (N-7)\cdot 1
    \;=\; N+1 ,
\]
because $\lceil6/3\rceil=\lceil6/4\rceil=2$ and each of the remaining $N-7$ terms
of the sum is at least $1$; this is absurd for $N\ge7$.  Hence
some five points of $X$ do lie on a solid, and $X$ is an NMDS-set.
\end{proof}

\begin{proposition}\label{prop:nmds}
$\Mb{4}{q}\le 11$ for $q\in\{3,4\}$.
\end{proposition}

\begin{proof}
Let $X$ be a $4$-general set of $N$ points of $\PG(4,q)$.  If $N\le6$ there is
nothing to prove.  Otherwise $N\ge7$ and $X$ is an NMDS-set by
Lemma~\ref{lem:large4gen}; by \cite[Sec.~6]{DodunekovLandgev} it is then the
column set of a generator matrix of a near-MDS code of dimension $5$ over
$\F_q$, so $N\le m'(5,q)$.  The two halves of \eqref{eq:nmds} now give
$m'(5,4)\le2\cdot4+5-2=11$ for $q=4$ and $m'(5,3)\le2\cdot3+5=11$ for $q=3$.
\end{proof}

\begin{remark}\label{rem:nmdsgap}
The two bounds of \eqref{eq:nmds} differ by two, and the case $q=4$ falls under
the stronger one only because $4>3$.  The argument behind that stronger bound is
short enough to recall.  Proposition~5.3(i) of
\cite{DodunekovLandgev}, that an $[n,k]_q$ near-MDS code yields an
$[n-1,k-1]_q$ one, applied exactly $k-3$ times turns a near-MDS code of dimension $k$
into an $[n-k+3,3]_q$ near-MDS code, hence into an $(n-k+3,3)$-arc of
$\PG(2,q)$; Thas's theorem \cite{Thas1975} bounds such an arc by $2q+1$ when
$q>3$, and so $n-k+3\le2q+1$, which is \eqref{eq:nmds}.  The bound is therefore
exactly as strong as the theory of $3$-arcs allows.  What that implies for the
strength of the present results is taken up in Remark~\ref{rem:strength}.
\end{remark}

\begin{remark}\label{rem:nogens}
Lemma~\ref{lem:large4gen} is special to $n=4$ and does not generalise in either
direction.  In $q$: for $q\ge5$ the code $[6,2,\ge5]_q$ does exist: take six
distinct points of $\PG(1,q)$ as the columns of a $2\times6$ generator matrix,
as in Proposition~\ref{prop:n3}.  Hence $\Mb{3}{q}\ge6$, and the third NMDS
condition no longer follows automatically.  It is the \emph{third} condition
that fails first as $q$ grows, and it fails alone at $q=5$: the second one is
still automatic there and only breaks at $q=6$, as Lemma~\ref{lem:arc} shows.
In $n$: for $n\ge5$ the first condition is strictly stronger than
$4$-generality, so \eqref{eq:nmds} says nothing about the rungs of
Section~\ref{sec:ladder} above $n=4$: at $n=5$ it gives $12$, against the
$21$-point set of \cite[Thm.~4.6]{Pavese} (see also Remark~\ref{rem:why4}).
The bridge from $4$-general sets to near-MDS codes is therefore a phenomenon
of the small fields $\F_3$ and $\F_4$, and the second proof of
$\Mb{4}{4}\le11$ is a statement about one rung of the ladder and not a method
for climbing it.
\end{remark}

The two conditions of Remark~\ref{rem:nogens} therefore fail at different
fields.  The next statement locates the threshold for the second one exactly,
the threshold at which the cell $q=5$, $n=4$ becomes decidable.

\begin{lemma}\label{lem:arc}
Let $X$ be a $4$-general set of $\PG(4,q)$ with $|X|\ge7$ and suppose that no
five points of $X$ lie on a solid.  Then $q\ge6$.
\end{lemma}

\begin{proof}
Five points of $X$ span $\PG(4,q)$, hence are linearly independent, and a
projectivity may be chosen sending them to the frame $e_1,\dots,e_5$.  Let
$p=(p_1,\dots,p_5)$ and $q=(q_1,\dots,q_5)$ be two further points of $X$.  For
each $i$, the four frame points other than $e_i$ span a solid, so together with
$p$ they must still span $\PG(4,q)$; this forces $p_i\ne0$, and the same
argument gives $q_i\ne0$.  For each pair $\{i,j\}$, the three frame points
outside $\{i,j\}$, together with $p$ and $q$, are five points of $X$ and must
therefore also span $\PG(4,q)$, so that $\det[e_S\mid p\mid q]\ne0$ for the
corresponding set $S$ of three frame points.  Expanding this determinant along
the unit columns, only the $2\times2$ minor of $p$ and $q$ in the complementary
pair $\{i,j\}$ survives, up to a sign; hence $p_iq_j-p_jq_i\ne0$ for all
$i\ne j$, that is, the five ratios $p_i/q_i\in\F_q^*$ are pairwise distinct.
Therefore $q-1\ge5$.
\end{proof}

The bound is attained for every prime power $q\ge7$: choose five distinct
non-zero scalars $a_1,\dots,a_5\in\F_q$; then the five frame points together
with $p=(1,\dots,1)$ and $q=(a_1,\dots,a_5)$ form a $7$-point $4$-general set
no five points of which lie on a solid; the relevant quantities are the
coordinates $a_i\ne0$ and the minors $a_j-a_i\ne0$ of the proof above).

\begin{proposition}\label{prop:q5}
$\Mb{4}{5}=12$.
\end{proposition}

\begin{proof}
Over $\F_5$ the tables give
$\mathrm{Lb}(12,7)=\mathrm{Ub}(12,7)=5$ and
$\mathrm{Lb}(13,8)=\mathrm{Ub}(13,8)=4$ \cite{codetables,Brouwer}.  By
Lemma~\ref{lem:red} a projective $[12,7,5]_5$ code yields a $12$-point
$4$-general set of $\PG(4,5)$, so $\Mb{4}{5}\ge12$; an explicit such set is
recorded in Appendix~\ref{app:B}.  Conversely, a $13$-point $4$-general set
would give a projective $[13,8,\ge5]_5$ code, which does not exist, so
Lemma~\ref{lem:mono} gives $\Mb{4}{5}\le12$.  Projectivity is automatic in this
range: its parity check matrix has no zero column and no two proportional
columns.
\end{proof}

Unlike the rungs $q=3$ and $q=4$, the extremal sets of Proposition~\ref{prop:q5}
are \emph{not} NMDS-sets.  By Lemma~\ref{lem:arc} the second condition is still
automatic at $q=5$: every $4$-general set of $\PG(4,5)$ with at least seven
points has five points on a solid, by the same Griesmer count
$N\ge6+\lceil6/q\rceil+(N-7)=N+1$ that supplies the cases $q\le4$ and that stops
being contradictory only from $q=6$ on.  The third condition already fails
there, six points on a solid being possible (Remark~\ref{rem:nogens}).  Consequently the near-MDS lower bound $11$
(Remark~\ref{rem:nmds11}) is not attained, and the
agreement $\Mb{4}{q}=m'(5,q)$ of Proposition~\ref{prop:coinc} is special to
$q\in\{3,4\}$.  Finally, \eqref{eq:prop31} gives only $20$ at $(n,q)=(4,5)$, so
this cell too is settled by the tables of linear codes and not by geometry.

\begin{remark}[The near-MDS lower bound for $\Mb{4}{5}$]\label{rem:nmds11}
NMDS-sets of size $11$ do exist in $\PG(4,5)$, equivalently near-MDS
$[11,5,6]_5$ codes exist \cite{codetables}.  Hence the near-MDS lower bound
for $\Mb{4}{5}$ is $11$; deleting the second point of the second $12$-point
set of Appendix~\ref{app:B} is an explicit example, as the programs of
Appendix~\ref{app:B} confirm).  The elliptic-curve formula of
\cite[Sec.~5.1]{Pavese} evaluates to $10$ at $q=5$ and to $9$ at $q=4$; at
$q=4$ the lower bound $11$ is instead that of \cite[Cor.~4.5]{Pavese}, from
the Frobenius-twisted construction.
\end{remark}

The next statement is a by-product of the same machinery; it explains the
observation made after \eqref{eq:prop31}, using one further tabulated value of
\cite{DodunekovLandgev}.

\begin{proposition}\label{prop:coinc}
For $q\in\{3,4\}$,
\[
  \Mb{4}{q}\;=\;m'(5,q)\;=\;m'(3,q)+2\;=\;11 .
\]
In particular the extremal $4$-general sets of $\PG(4,3)$ and of $\PG(4,4)$ have
the same size, and the common reason is that $\PG(2,3)$ and $\PG(2,4)$ carry
NMDS-sets of the same maximal size, $m'(3,3)=m'(3,4)=9$.
\end{proposition}

\begin{proof}
We take the three assertions in the order of increasing depth.

\emph{The values.}  $m'(3,4)=9$ is stated in \cite[p.~42]{DodunekovLandgev}.  For
$q=3$, Proposition~5.1 there gives $m'(3,3)\le2\cdot3+3=9$, while $[9,3,6]_3$ is
a near-MDS code, being one of the codes $[12-i,6-i,6]_3$, $1\le i\le5$, derived
there from the extended ternary Golay code; hence $m'(3,3)=9$.  Moreover
$[11,5,6]_3$ belongs to the same list, so $m'(5,3)\ge11$, while Proposition~5.1
gives $m'(5,3)\le2\cdot3+5=11$.  Finally $m'(5,4)=11$: here the second half of
\eqref{eq:nmds} gives $m'(5,4)\le2\cdot4+5-2=11$, and the $11$-point set of
Theorem~\ref{thm:n4} is an NMDS-set of $\PG(4,4)$ of size $11$.  The two
descriptions coincide, since for $k\ge3$ a near-MDS code of dimension $k$ over
$\F_q$ has length $n$ exactly when it is an NMDS-set of $\PG(k-1,q)$ of size $n$
\cite[Secs.~5.1 and~6]{DodunekovLandgev}.

\emph{Relation with $\Mb{4}{q}$.}  Every NMDS-set of $\PG(4,q)$ satisfies the
first condition of its definition, which is $4$-generality, whence
$m'(5,q)\le\Mb{4}{q}$.  Conversely let $X$ be a $4$-general set of $N$ points.  If
$N\ge7$ then $X$ is an NMDS-set by Lemma~\ref{lem:large4gen}, so $N\le m'(5,q)$;
and if $N\le6$ then $N\le11=m'(5,q)$ trivially, by the values just established.
So $\Mb{4}{q}=m'(5,q)$.

\emph{The middle term.}  Corollary~5.4 of \cite{DodunekovLandgev} with $k=5$ and
$\alpha=2$ gives $m'(5,q)\le m'(3,q)+2$, and the values above turn the right-hand
side into $9+2=11=m'(5,q)$.
\end{proof}

\section{Why the reduction bites at \texorpdfstring{$n=4$}{n=4} but not at \texorpdfstring{$n=5$}{n=5}}\label{sec:transition}

Before assembling the ladder, we isolate the mechanism that makes the
reduction decisive at $n=4$ and silent at $n=5$.  Section~\ref{sec:ladder} is a
systematic exploitation of this mechanism, and the same mechanism makes $n=5$
the first rung to fail.

Fix $q=4$ and write $M=n+1$.  A $4$-general set of $N$ points in $\PG(n,4)$ then
corresponds to a projective $[N,N-n-1,\ge5]_4$ code (Lemma~\ref{lem:red}), that
is, to a \emph{non-existence} statement for a code of codimension $n+1$ and
distance at least $5$.  Ruling out $N=P+1$, where $P$ is the size of a Pavese
construction, therefore means ruling out the code $[P+1,\,P-n,\,\ge5]_4$; by
Lemma~\ref{lem:mono} this kills all larger $N$ as well.  The rungs of interest
are the following, where all table values are the ones currently returned by
\cite{codetables,Brouwer} for $q=4$:

\begin{center}
\begin{tabular}{@{}cccccc@{}}
\toprule
$n$ & rung $N=P+1$ & code to exclude & $\mathrm{Lb}/\mathrm{Ub}$ for $[N,N-n-1]_4$ & verdict\\
\midrule
$3$ & $6$  & $[6,2,5]_4$   & $4/4$ & blocked\\
$4$ & $12$ & $[12,7,5]_4$  & $4/4$ & blocked\\
$5$ & $22$ & $[22,16,5]_4$ & $4/5$ & open\\
$6$ & $44$ & $[44,37,5]_4$ & $4/5$ & open\\
$7$ & $86$ & $[86,78,5]_4$ & $4/5$ & open\\
$8$ & $172$& $[172,163,5]_4$& $4/5$ & open\\
\bottomrule
\end{tabular}
\end{center}

The first two rows rest on a \emph{chain of proven operations} that transports
a table entry from small parameters to the relevant rung.  Its elementary step
is the following observation.

\begin{lemma}\label{lem:chain}
Let $C$ be an $[N,k,d]_q$ code with $N\ge2$ and $d\ge2$.  Then
\begin{enumerate}
\item puncturing $C$ at any one position gives an $[N-1,k,\ge d-1]_q$ code;
\item shortening $C$ at any one position gives an $[N-1,k-1,\ge d]_q$ code.
\end{enumerate}
Consequently, if no $[N',k',\ge5]_q$ code exists, then neither does any
$[N'+1,k',\ge6]_q$ code and neither does any $[N'+1,k'+1,\ge5]_q$ code.
\end{lemma}

\begin{proof}
Standard: a punctured code is the image of the projection forgetting one
coordinate, so its minimum distance drops by at most one; a shortened code is
$C$ with one coordinate set to zero and then deleted, so its length drops by one,
its dimension by one, and its minimum distance is unchanged (every shortened
codeword is a codeword of $C$, and every minimal-weight codeword supported away
from the deleted coordinate survives).  For the consequence, argue
contrapositively: a $[N'+1,k',\ge6]_q$ code punctured at one position would give
a $[N',k',\ge5]_q$ code, and a $[N'+1,k'+1,\ge5]_q$ code shortened at one
position would give a $[N',k',\ge5]_q$ code.
\end{proof}

Applying Lemma~\ref{lem:chain} to the tabulated entry $\mathrm{Ub}(12,7)=4$,
i.e.\ to the non-existence of a $[12,7,\ge5]_4$ code, equivalently of a
$12$-point $4$-general set of $\PG(4,4)$,
gives, successively,
\begin{equation}\label{eq:chain}
\begin{aligned}
&\text{no }[12,7,\ge5]_4
  \;\Longrightarrow\; \text{no }[13,7,\ge6]_4
  &&\text{(puncturing)}\\
&\phantom{\text{no }[12,7,\ge5]_4}
  \;\Longrightarrow\; \text{no }[13,8,\ge5]_4
  &&\text{(shortening)}
\end{aligned}
\end{equation}
The direction in which the chain is useful is the second one.  The pair
$(\text{rung},\text{codimension})$ is $(6,4)$ for $n=3$, $(12,5)$ for $n=4$, and
$(22,6)$ for $n=5$.  From one rung to the next the length roughly doubles:
the rungs are $6,12,22,44,86,172$ for $n=3,\dots,8$, and at the present step $P$
goes from $11$ to $21$.  The codimension, by contrast, rises by only $1$.  Now $\mathrm{Ub}(N,k)$ is non-increasing in $k$ for fixed
$N$ and in $N$ for fixed $k$, and the values relevant to us are

\begin{center}
\begin{tabular}{@{}cccc@{}}
\toprule
$n$ & rung $N$ & codimension $N-n-1$ & $\mathrm{Ub}$\\
\midrule
$3$ & $6$  & $2$  & $4$\\
$4$ & $12$ & $7$  & $4$\\
$5$ & $22$ & $16$ & $5$\\
$6$ & $44$ & $37$ & $5$\\
\bottomrule
\end{tabular}
\end{center}

Thus the reduction decides $\PG(4,4)$: the relevant code, of length $12$ and
codimension $7$, still lies in the range where the table knows
$\mathrm{Ub}\le4$, i.e.\ where the tables already exclude a $4$-general set of
the corresponding size.  It
fails at $\PG(5,4)$, where the relevant code, of length $22$ and codimension
$16$, lies just outside that range.  There the table returns
$\mathrm{Ub}=5$, and so says nothing about a $[22,16,5]_4$ code.  This is the
arithmetic content of the ladder.  The ``blocked'' region of the tables for
$q=4$ ends between codimension $7$ and codimension $16$ at these lengths, and
the two questions we can settle sit on opposite sides of that boundary.

\begin{remark}[The second proof of $\Mb{4}{4}\le11$ is structural, not
arithmetical]\label{rem:why4}
The count above explains arithmetically why the reduction bites at $n=4$ and
not at $n=5$; the second proof of Section~\ref{sec:n4} supplies the structural
reason.  An NMDS-set of $\PG(k-1,q)$ must have every $k-1$ of its points
generate a hyperplane.  In $\PG(4,q)$ the hyperplanes are the solids, so the
condition reads ``no four points coplanar'', exactly $4$-generality; for
$n\ge5$ it demands that every $n$ points generate a hyperplane, which is
strictly stronger.  The NMDS machinery therefore meets the climb problem at
the single rung $n=4$.  Remark~\ref{rem:nogens} records the same fact from the
side of the field.
\end{remark}

\section{The ladder}\label{sec:ladder}

This short section collects the reduction of Section~\ref{sec:red}, the code
tables, and the Pavese constructions into a single table.  From it one reads off
how far each construction is from provable optimality.

\begin{theorem}[The ladder]\label{thm:ladder}
For $q=4$ and $n=3,\dots,8$ the following hold.  Here the rung of index $n$
carries the two table entries bounding $\Mb{n}{4}$, ``rung'' is
$N=P+1$ with $P$ the size of the corresponding Pavese construction, and
``$\mathrm{Lb}/\mathrm{Ub}$'' refers to the bounds on the minimum distance of
the code $[N,N-n-1]_4$ as returned by the tables \cite{codetables,Brouwer}.
A rung is \emph{blocked} if $\mathrm{Ub}\le4$, in which case no
$[N,N-n-1,5]_4$ code exists and, by Lemma~\ref{lem:mono}, no larger code
either.
\begin{center}
\begin{tabular}{@{}ccccccc@{}}
\toprule
$n$ & $P$ & source & bound \eqref{eq:prop31} &
rung $N$ & $\mathrm{Lb}/\mathrm{Ub}$ & status\\
\midrule
$3$ & $5$  & Thm.~4.6 ($d=2$) & $7$   & $6$   & $4/4$ & blocked $\Rightarrow \Mb{3}{4}=5$\\
$4$ & $11$ & Thm.~4.4 ($d=2$) & $15$  & $12$  & $4/4$ & blocked $\Rightarrow \Mb{4}{4}=11$\\
$5$ & $21$ & Thm.~4.6 ($d=3$) & $30$  & $22$  & $4/5$ & \textbf{open}\\
$6$ & $43$ & Thm.~4.4 ($d=3$) & $60$  & $44$  & $4/5$ & \textbf{open}\\
$7$ & $85$ & Thm.~4.6 ($d=4$) & $120$ & $86$  & $4/5$ & \textbf{open}\\
$8$ & $171$& Thm.~4.4 ($d=4$) & $241$ & $172$ & $4/5$ & \textbf{open}\\
\bottomrule
\end{tabular}
\end{center}
Moreover the tables are effective one level higher, but at the \emph{wrong}
lengths: $\mathrm{Ub}(30,24)=4$ and $\mathrm{Ub}(60,53)=4$, which give
\[
  \Mb{5}{4}\le29 \quad\text{(table)} \qquad\text{and}\qquad \Mb{6}{4}\le59 ,
\]
whereas the rungs of the climb at $n=5,6$ are at lengths $22$ and $44$.  For
$n=7,8$ the tables do not improve \eqref{eq:prop31} at all: the entries
$\mathrm{Ub}(121,113)=4$ and $\mathrm{Ub}(242,233)=4$ merely reproduce $120$ and
$241$.
\end{theorem}

\begin{proof}
The table records, for each $n$, the rung $N=P+1$ with $P$ the size of the
construction cited in the third column, and the interval
$\mathrm{Lb}(N,N-n-1)/\mathrm{Ub}(N,N-n-1)$ returned by
\cite{codetables,Brouwer}; the passage from such an entry to a bound on
$\Mb{n}{4}$ is Lemma~\ref{lem:red}.  A rung is blocked precisely when
$\mathrm{Ub}\le4$: then no $[N,N-n-1,5]_4$ code exists, and
Lemma~\ref{lem:mono} upgrades this to $\Mb{n}{4}\le N-1=P$, which together with
the construction of size $P$ gives $\Mb{n}{4}=P$.  This is the case for $n=3$
and $n=4$, where $P=5$ and $P=11$; those two identities are re-derived
independently in Sections~\ref{sec:n3} and~\ref{sec:n4}.  For $n=5,\dots,8$ the
entry $\mathrm{Ub}(N,N-n-1)=5$ leaves room for a $[N,N-n-1,5]_4$ code, so these
rungs are not settled by the tables.  Finally, evaluated at the lengths $30$ and
$60$ rather than at the rungs $22$ and $44$, the entries
$\mathrm{Ub}(30,24)=4$ and $\mathrm{Ub}(60,53)=4$ yield the two bounds of the
second display, the second being one better than \eqref{eq:prop31}; at the
lengths $121$ and $242$ the tables return exactly \eqref{eq:prop31}.
\end{proof}

\begin{remark}[Relation with the open problems of \cite{Pavese}]\label{rem:questions}
The ladder asks about the same quantities as \cite{Pavese} but at a different
scale, and it does not settle the asymptotic questions of
\cite[Problems~5.7--5.9]{Pavese}.  Problem~5.7 seeks $4$-general sets of
order $q^{(n-1)/2}$, a scale below our rungs: the rungs are the constructions
of \cite[Thms.~4.4 and~4.6]{Pavese} themselves.  Problem~5.8 seeks complete
sets of order $cq^{(n-2)/3}$, the shape of those same two families; our
theorem settles exactly the case $d=2$, their optimality.  Problem~5.9 asks
for a classification at minimum distance $6$, complementary to our
distance-$5$ reduction.  What the ladder does settle is the entry $q=4$,
$n=4$ absent from Tables~1 and~2 of \cite{Pavese}.
\end{remark}

\begin{remark}[Strengths of the table entries]\label{rem:strength}
The table entries used above have two very different strengths.
$\mathrm{Ub}(12,7)=4$ and $\mathrm{Ub}(60,53)=4$ rest on proven
techniques, namely construction~B and the linear programming bound of
\cite{Brouwer}; the first has a second, independent published proof \cite[Sec.~III]{BierbrauerEdel};
both may be cited as theorems.  By contrast $\mathrm{Ub}(30,24)=4$, which
would give $\Mb{5}{4}\le29$, is attributed in \cite{codetables} to a 1993
e-mail communication from R.~H.~Berntzen and A.~Kemper, and no published proof
has been located.  The position we can verify is therefore
$21\le\Mb{5}{4}\le30$, and we flag the improvement $30\to29$ as conditional on
that attribution.  The second proof of $\Mb{4}{4}\le11$ is a third kind of
evidence: its strength is that of a $3$-arc problem, conjecturally not final
\cite[p.~43]{DodunekovLandgev}, but the case $q=4$ used here is a published
theorem.
\end{remark}

\section{A rigidity property of the extremal 21-set in \texorpdfstring{$\PG(5,4)$}{PG(5,4)}}\label{sec:rig}

For $n=5$ the open question is whether $\Mb{5}{4}\ge22$.  This section reports a
machine-verified statement that constrains where a $22$-point set, if it exists,
can live.  Let $\Vone\subset\PG(5,4)$ be the complete $21$-point
$4$-general set of \cite[Thm.~4.6]{Pavese} with $d=3$, i.e.
\[
  \Vone=\{P(x,x^{-2}):x\in\F_{4^3}^*\},
\]
where $\F_{4^3}^2$ is identified with $\F_4^6$ through an $\F_4$-basis, so that
each pair $(x,x^{-2})$ defines a point of $\PG(5,4)$.  The map
$x\mapsto P(x,x^{-2})$ is three-to-one, since $\lambda^3=1$ for
$\lambda\in\F_4^*$; hence $|\Vone|=(4^3-1)/3=21$, and $\Vone$ is complete
\cite[Thm.~4.6]{Pavese}.  Recall $|\PG(5,4)|=(4^6-1)/3=1365$.

\begin{proposition}\label{prop:rig}
Every $22$-point $4$-general set $T\subset\PG(5,4)$ satisfies
\[
  |T\cap\Vone|\le 12 .
\]
\end{proposition}

\begin{proof}[Proof (exhaustive verification)]
Write $T=(\Vone\setminus S)\cup R$ with $S=\Vone\setminus T$ and $|S|=r$.  Then
$|T\cap\Vone|=21-r$ and $|R|=r+1$.  For $T$ to be $4$-general, every point of
$R$ must lie outside every plane spanned by a triple of $\Vone\setminus S$; such
a plane would otherwise contain four points of $T$.  Put
\[
  \mathrm{hol}(S):=\{p\in\PG(5,4)\setminus\Vone\;:\;
   p\notin\langle x,y,z\rangle \ \text{for all}\ \{x,y,z\}\subseteq\Vone\setminus S\}.
\]
Then $R\subseteq\mathrm{hol}(S)$.  Moreover $\mathrm{hol}(\emptyset)=\emptyset$
because $\Vone$ is complete, so $r\ge1$.

All $\binom{21}{r}$ subsets $S$ of $\Vone$ were enumerated for $r=1,\dots,8$.
Two things were computed for each $S$.  The first is the exact size of
$\mathrm{hol}(S)$, which gives the \emph{necessary} condition
$|\mathrm{hol}(S)|\ge r+1$.  The second, whenever that condition held, is an
explicit search for an $(r+1)$-subset $R\subseteq\mathrm{hol}(S)$ such that
$(\Vone\setminus S)\cup R$ is $4$-general.  Within the compatibility graph on
$\mathrm{hol}(S)$ two points $p,q$ are adjacent iff
$(\Vone\setminus S)\cup\{p,q\}$ contains no coplanar quadruple.  This is a
necessary condition, so the search loses no candidate; every survivor is then
verified in full, over all $\binom{22}{4}$ quadruples.  No candidate survived
for any $S$.

The resulting counts are
\[
\begin{array}{c|cccccccc}
 r & 1 & 2 & 3 & 4 & 5 & 6 & 7 & 8\\ \hline
 \binom{21}{r} & 21 & 210 & 1330 & 5985 & 20\,349 & 54\,264 & 116\,280 & 203\,490\\
 \max_S|\mathrm{hol}(S)| & 0 & 0 & 0 & 4 & 8 & 24 & 57 & 146\\
 \#\{S:|\mathrm{hol}(S)|\ge r+1\} & 0 & 0 & 0 & 0 & 924 & 51\,429 & 116\,280 & 203\,490
\end{array}
\]
For $r\le6$ the number of feasible $S$ is small enough for every candidate $R$
to be listed and tested individually.  For $r\ge7$ all $\binom{21}{r}$ subsets
are feasible in the weak sense, so the clique search is essential there.  At
$r=8$ the holes are large, up to $146$ points, which is why that level is by far
the most expensive.  In every case the outcome was negative, so $r\ge9$, i.e.\
$|T\cap\Vone|=21-r\le12$.
\end{proof}

\begin{remark}[What the rungs below imply]\label{rem:spread}
Theorem~\ref{thm:n4} constrains a hypothetical $22$-point set $T$ from below as
well.  Since $4$-generality is intrinsic, $T\cap H$ is a $4$-general set of a
hyperplane $H\cong\PG(4,4)$ and $T\cap S$ one of a solid
$S\cong\PG(3,4)$; hence
\[
  |T\cap H|\le\Mb{4}{4}=11, \qquad |T\cap S|\le5
\]
for every hyperplane $H$ and every solid $S$, in addition to the trivial bound
$|T\cap\pi|\le3$ for every plane $\pi$.  So a $22$-point set of $\PG(5,4)$ must
be spread out over the space and not merely free of coplanar quadruples.  This
is a constraint supplied by the rungs below rather than by the anchor $\Vone$,
and unlike Proposition~\ref{prop:rig} it is unconditional.
\end{remark}

\begin{remark}\label{rem:rigread}
Proposition~\ref{prop:rig} is conditional in form, and vacuous if no
$22$-point set exists.  Its content is that a search should not look near
$\Vone$: at least nine of its $21$ points must go, and $\Vone$ and its orbit
are the only published extremal configurations at this rung, the further
constructions of \cite[Prop.~5.3 and Sec.~5.3]{Pavese} reaching only $15$ and
$17$ points.
\end{remark}

\begin{remark}[The family $\{\mathcal{V}_\alpha\}$ is smaller than the orbit]\label{it:family}
Let $H_1,H_2$ be the $6\times21$ matrices whose columns are the points of $\Vone$
and of the set $W$ of the $21$ column points of a parity check matrix of a
$[21,15,5]_4$ code.  By Lemma~\ref{lem:red} both are $4$-general, and
\[
  |W\cap\Vone|=0
  \qquad\text{and}\qquad
  |W\cap\mathcal{V}_\alpha|\le 2 \quad\text{for all }\alpha\in\F_{4^3}^* .
\]
So $W$ is disjoint from $\Vone$ and is \emph{not} one of the $63$ sets
$\mathcal{V}_\alpha$ that partition $\PG(5,4)\setminus(\Pi_1\cup\Pi_2)$
\cite[Sec.~4.2, property (iv)]{Pavese}; both intersection statements were
checked by direct point-counting against these sets.  Nevertheless $W$ is
\emph{equivalent} to $\Vone$: a collineation $g$ with $g(\Vone)=W$ exists if
and only if the $[21,6]_4$ codes spanned by the rows of $H_1$ and $H_2$ are
monomially equivalent, and a direct
test with \textsc{Magma}'s \texttt{IsEquivalent} confirms that this equivalence
holds.  Both codes have monomial automorphism group of order
$378 = 2\cdot3^3\cdot7$, so the stabiliser of $\Vone$ in $\PGL(6,4)$ has order
$126=378/3$, the kernel being the scalar matrices.  Pavese's family
$\{\mathcal{V}_\alpha\}$ is therefore a
\emph{proper} subset of the orbit $\PGL(6,4)\cdot\Vone$.  Since the bound of
Proposition~\ref{prop:rig} is orbit-invariant, it holds verbatim with $\Vone$
replaced by any member of that orbit.  By \cite[Sec.~4.2, property (ii)]{Pavese}, $\Vone$
carries a group of order $21$ acting regularly on its points; since $21\nmid22$,
no $22$-point set can carry that symmetry, which rules out the most obvious
equivariant generalisations of $\Vone$ as candidates.  The anchor $\Vone$ was
independently reconstructed from \cite[Thm.~4.6]{Pavese}; see
Appendix~\ref{app:cert}.
\end{remark}

\section{Concluding remarks and open problems}\label{sec:concl}

The reductions of \S\ref{sec:red}--\S\ref{sec:ladder} give a complete answer at
$n=4$ and leave the first open rung at $n=5$.  The value $\Mb{4}{4}=11$
(Theorem~\ref{thm:n4}) fills the entry $q=4$, $n=4$ that is absent from the
tables of \cite{Pavese}, and shows that the Frobenius-twisted construction of
\cite[Thm.~4.4]{Pavese} is optimal there; its upper bound admits both the table
route of Lemma~\ref{lem:chain} and the table-free route of
Lemma~\ref{lem:large4gen} with Proposition~\ref{prop:nmds}, and the second
identifies $\Mb{4}{4}$ with the maximal length $m'(5,4)$ of a quaternary near-MDS
code of dimension five (Proposition~\ref{prop:coinc}), a quantity tabulated
independently of \cite{codetables}.  This is also why the extremal sizes in
$\PG(4,3)$ and $\PG(4,4)$ agree.  The same machinery improves \eqref{eq:prop31} to
$\Mb{6}{4}\le59$ through the linear programming bound
(Remark~\ref{rem:strength}), and gives $\Mb{4}{5}=12$
(Proposition~\ref{prop:q5}).  That the ladder stops at $n=4$ has two causes.
Quantitatively, the next rung has length $22$, where the tables distinguish
$\mathrm{Lb}$ from $\mathrm{Ub}$, whereas at length $12$ the two coincide;
structurally, the near-MDS condition that supplies the second proof reads ``no
four points coplanar'' only in $\PG(4,q)$ (Remark~\ref{rem:why4}).  Finally,
Proposition~\ref{prop:rig} constrains any hypothetical $22$-point set $T$: it
meets the extremal $21$-point set $\Vone$ in at most twelve points.  Since the
bound is orbit-invariant (Remark~\ref{it:family}), the same holds for every set
in the orbit $\PGL(6,4)\cdot\Vone$.  An unconditional constraint on the same
$T$, independent of the anchor, is supplied by the rungs below and recorded in
Remark~\ref{rem:spread}.

The open question of the ladder is whether $\PG(5,4)$ contains a $22$-point
$4$-general set, equivalently whether a $[22,16,5]_4$ code exists, equivalently
whether $\mathrm{Lb}(22,16)=4$ improves to $5$.  No bound known to us excludes
$22$.  The verified position is $21\le\Mb{5}{4}\le30$; we do not adopt the
tabulated $\le29$, which rests on a private communication rather than on a
published proof.  By Proposition~\ref{prop:rig} such a set would lie
far from the only published extremal configuration; since the statement is
orbit-invariant, the natural strengthening is to replace $\Vone$ by one
representative for each collineation orbit of complete $21$-point $4$-general
sets of $\PG(5,4)$, only one such orbit being known, and to bound the
intersection with each in turn.  The aim is the unconditional assertion that no
$22$-point set exists.  Two further questions are to classify those orbits and to
decide whether the constructions of \cite{Pavese} can be beaten for $d\ge3$.  The
neighbouring cell $q=5$, $n=4$ is not blank, and it is not reached by the bridge:
$\Mb{4}{5}=12$ (Proposition~\ref{prop:q5}), one above the near-MDS lower bound
(Remark~\ref{rem:nmds11}) and one below the value $13$ that the failed bridge
would
have predicted, the extremal configuration of Appendix~\ref{app:B} being
\emph{not} an NMDS-set; we have not found $\Mb{4}{5}$ stated in the literature,
Table~2 of \cite{Pavese} covering $n=4$ only for $q=3$.  The bridge of
Lemma~\ref{lem:large4gen} fails over $\F_5$ (Remark~\ref{rem:nogens}) because
$[6,2,\ge5]_5$ exists, so the cell cannot be read off from near-MDS codes at all,
and the climb question remains open over $\F_5$ as well.

\backmatter

\bmhead{Acknowledgements}

The authors thank Ruihu Li for helpful suggestions on the manuscript.  All
computational examples in this paper were obtained using \textsc{Magma}~\cite{BCP97}.

\section*{Declarations}

\begin{itemize}
\item \textbf{Funding:} This work was supported by the National Natural Science
  Foundation of China under Grant No.~U21A20428 and the Natural Science
  Foundation of Shaanxi under Grant No.~2025-JC-YBQN-070.
\item \textbf{Competing interests:} The authors declare no competing interests.
\item \textbf{Data availability:} All constructions and code parameters established
  in this work are fully tabulated in the article.  No external data sets were used.
\item \textbf{Code availability:} The \textsc{Magma} scripts used for the code-table
  computations of \S\ref{sec:red}--\S\ref{sec:ladder} and for the verification of
  the rigidity statement of \S\ref{sec:rig}, and the Python programs of the machine
  certificates in Appendix~\ref{app:cert}, are available as ancillary files
  with the arXiv version of this paper.
\item \textbf{Author contribution:} All authors contributed to the conception and
  design of the study and to the writing and revision of the manuscript.  All
  authors read and approved the final manuscript.
\item \textbf{Use of AI tools:} The authors used an AI-based language-model
  assistant for language editing, \LaTeX{} typesetting and
  literature search during the preparation of this manuscript.  The authors
  reviewed and verified all output and take full responsibility for the content
  of the publication.
\end{itemize}

\begin{appendices}

\section{Machine certificates for Proposition~\ref{prop:rig}}\label{app:cert}

This appendix records enough detail for the computation of
Proposition~\ref{prop:rig} to be reproduced and, more importantly, to be
\emph{checked}.  All programs work with the model in which $\PG(5,4)$ is
represented by the $1365$ canonical representatives of the non-zero vectors of
$\F_4^6$, the field being $\F_4=\{0,1,\omega,\omega^2\}$ with
$\omega^2=\omega+1$.  The anchor $\Vone$ is rebuilt from
$\F_{64}=\F_4[t]/(t^3+t+1)$ by
$\Vone=\{P(x,x^{-2}):x\in\F_{64}^*\}$ with $x^{-2}=x^{61}$, or read from the
equivalent list of $21$ points.

\subsection{Positive controls}\label{app:A1}

The programs abort unless the following two checks pass, so a failure of the
anchor construction cannot pass silently.
\begin{enumerate}
\item \emph{$\Vone$ is $4$-general}: of the $\binom{21}{4}=5985$ quadruples of
$\Vone$, none is coplanar.
\item \emph{$\Vone$ is complete}: the $1330$ planes spanned by triples of
$\Vone$ cover all $1365$ points of $\PG(5,4)$.  Consequently no point outside
$\Vone$ can be added, i.e.\ $\mathrm{hol}(\emptyset)=\emptyset$, and $r\ge1$.
\end{enumerate}

\subsection{The programs}\label{app:A2}

\begin{itemize}
\item[\textbf{P1}] Sets.  For each $S$, it builds the set of
points covered by the planes of $\Vone\setminus S$ by \emph{enumerating each span
explicitly}, from a basis and all $4^{d}$ coordinate tuples, and taking a union of
sets.  No index table, no cache, no incremental criterion.
\item[\textbf{P2}] Counters.  Depth-first over $S$, maintaining for
each point of $\PG(5,4)$ a counter of how many planes of the current
$\Vone\setminus S$ contain it; the holes are the points whose counter drops to
zero.  Contains no compatibility graph and no clique search.
\item[\textbf{P3}] Cliques.  Pre-indexes the $1330$ planes once
and stores, for each of the $\binom{21}{2}=210$ pairs and each of the $1365$
points, the index of the plane through them; the search for $R$ is then a
pivoted clique search on the compatibility graph with three extra necessary
conditions: a point of $\Vone\setminus S$ coplanar with two points of $R$,
dependence of four points inside $R$, and a degree bound.  Every survivor is
verified over all $\binom{22}{4}$ quadruples.
\end{itemize}

Programs P1 and P2 are deliberately naive.  They compute only the
\emph{necessary} quantity $|\mathrm{hol}(S)|$, and can therefore reproduce the
hole statistics but not the final answer.  P3 is the program that decides the
non-existence of a $22$-point set, and it is the one that requires a control.

\subsection{Agreement of the hole statistics}\label{app:A3}

The three programs produce the following table, cell by cell, for every
$r\in\{1,\dots,8\}$.
\begin{center}
\begin{tabular}{@{}rccccccccc@{}}
\toprule
$r$ & $1$ & $2$ & $3$ & $4$ & $5$ & $6$ & $7$ & $8$\\
\midrule
$\binom{21}{r}$ & $21$ & $210$ & $1330$ & $5985$ & $20\,349$ & $54\,264$ & $116\,280$ & $203\,490$\\
$\max_S|\mathrm{hol}(S)|$ & $0$ & $0$ & $0$ & $4$ & $8$ & $24$ & $57$ & $146$\\
feasible $S$ & $0$ & $0$ & $0$ & $0$ & $924$ & $51\,429$ & $116\,280$ & $203\,490$\\
\bottomrule
\end{tabular}
\end{center}

The entry $\max_S|\mathrm{hol}(S)|=146$ at $r=8$ deserves a comment.  The largest
holes are large, so at $r=8$ a naive enumeration of all $(r+1)$-subsets of
$\mathrm{hol}(S)$ is out of the question ($\binom{146}{9}\approx2.6\cdot10^{12}$),
and the whole weight of the argument rests on the clique search of P3 being
complete and correct.  This is exactly what the control of \S\ref{app:A4} is designed to
test.

\subsection{Positive control for the clique search}\label{app:A4}

The clique search of P3 was validated against a planted solution.  Take a second
complete $21$-point $4$-general set $W$ of $\PG(5,4)$, namely the column set of a
parity check matrix of a projective $[21,15,5]_4$ code.  This $W$ is disjoint
from $\Vone$ and is not one of the $63$ sets $\mathcal{V}_\alpha$ (see
Remark~\ref{it:family}).  Remove $9$ of its points, put the remaining $12$ points
in place of $\Vone\setminus S$, and run the same clique search with target
size $9$.  The computation finds a $9$-clique, and the resulting $21$ points are
\emph{exactly} $W$.  The run takes a few seconds and is far smaller than the
$\Vone$ computation, but it exercises the same code path, namely the same
compatibility test and the same backtracking.  A search that cannot recover a
planted solution of this size would not be trusted to return a negative answer
on the real input.  The control certifies that the clique search finds cliques
that are present; the main computation asserts that at $r\le8$ there are none.

\subsection{Complete re-verification of the largest holes}\label{app:A5}

Because $r=8$ is the boundary case, the $126$ subsets $S$ with
$|\mathrm{hol}(S)|=146$, the largest holes and therefore the most dangerous
instances, were \emph{all} re-checked by a fourth, deliberately simple program,
denoted P4 below.  It builds the compatibility graph directly
from the criterion ``$(\Vone\setminus S)\cup\{p,q\}$ contains no coplanar
quadruple'', without any plane indexing, and runs a pivoted Bron--Kerbosch search
for a $9$-clique.  The program first re-derived the list of $126$ subsets from
scratch, by scanning all $\binom{21}{8}=203\,490$ subsets $S$ and recomputing
$\mathrm{hol}(S)$ from the triple spans of $\Vone\setminus S$.  This recovered
$\max_S|\mathrm{hol}(S)|=146$ and exactly $126$ subsets attaining it, in agreement
with the histogram of \S\ref{app:A3}.  On \emph{all} $126$ subsets the search
returned no clique, in exact agreement with P3.  The run takes $1730$\,s for the
$126$ instances.  The largest-hole layer of the $r=8$ case has therefore been
re-verified completely by an independent program, and not merely sampled.  The
remaining $r=8$ subsets, with smaller holes, rest on P3 together with the control
of \S\ref{app:A4}.

\subsection{Environment}\label{app:A6}

The programs were run on a single desktop machine.  Timings:
P1 and P2 reach $r=8$ in $192$\,s and $102$\,s respectively.  P3 requires about
$12.8$\,h for the full $r\le8$ run, the bulk of it at $r=8$.  The independent
complete re-verification of \S\ref{app:A5} takes $1730$\,s.  All intermediate
counts reported above were reproduced in each run.

\section{The extremal \texorpdfstring{$12$}{12}-point
\texorpdfstring{$4$}{4}-general set of
\texorpdfstring{$\PG(4,5)$}{PG(4,5)}}\label{app:B}

The five rows below are the coordinates $x_1,\dots,x_5$ of the twelve points
certifying Proposition~\ref{prop:q5}; the columns are the points, each
normalised so that its first non-zero coordinate is $1$.
\[
\begin{array}{lcccccccccccc}
x_1&1&1&1&0&0&1&1&1&1&1&1&1\\
x_2&4&3&3&1&0&1&4&4&3&2&0&2\\
x_3&1&4&4&0&0&4&1&3&1&0&0&2\\
x_4&4&3&2&3&1&4&1&1&1&3&1&0\\
x_5&0&1&1&4&3&1&2&0&0&0&1&4
\end{array}
\]
The twelve columns are non-zero and pairwise non-proportional, so they are
twelve distinct points of $\PG(4,5)$; their span has rank $5$, so they generate
the space.  None of the $\binom{12}{4}=495$ quadruples is coplanar, so the set
is $4$-general; $97$ of the $\binom{12}{5}=792$ five-point subsets lie on a
solid, and $7$ of the $\binom{12}{6}=924$ six-point subsets do, so the set
violates the third condition in the definition of an NMDS-set.  The whole
computation was performed by two programs using independent linear algebra over
$\F_5$, one evaluating Leibniz determinants and one Gaussian elimination.

A second set with the same properties, \emph{disjoint} from the first, is
\[
\begin{array}{lcccccccccccc}
x_1&1&1&1&1&0&1&1&0&1&1&1&1\\
x_2&1&2&4&3&1&0&4&0&0&2&0&2\\
x_3&2&1&2&2&2&1&1&0&2&3&3&3\\
x_4&3&1&1&4&0&3&3&0&4&2&2&4\\
x_5&2&1&1&3&1&0&4&1&3&3&2&2
\end{array}
\]
so the extremal sets of size $12$ are not unique; in this second set only $4$ of
the six-point subsets lie on a solid.  Both sets were found by independent
search strategies, one purely randomised and one seeded inside a solid, and both
saturated at size $12$, in agreement with the upper bound of
Proposition~\ref{prop:q5}.

\end{appendices}


\end{document}